\documentclass[a4paper]{article}

\title{Hardy-Carleman Type Inequalities for Dirac Operators}
\author{
Alexandra Enblom\\
{\small Department of Mathematics}\\
{\small Linköping University}\\
{\small SE-581 83 Link\"oping, Sweden}\\
{\small \texttt{alexandra.enblom@liu.se}}\\
}

\date{}

\usepackage[utf8]{inputenc}
\usepackage[T1]{fontenc}
\usepackage{amsmath}
\usepackage{amsthm}
\usepackage{amssymb} 
\usepackage{mathrsfs} 
\usepackage{enumerate}

\theoremstyle{plain}
\newtheorem{thm}{Theorem}[section]

\newtheorem{cor}[thm]{Corollary}

\theoremstyle{definition}

\newtheorem{rem}[thm]{Remark}

\newcommand{\reals}{\ensuremath{\mathbb{R}}}
\newcommand{\complex}{\ensuremath{\mathbb{C}}}

\newcommand{\ud}{\,d}

\numberwithin{equation}{section}

\begin{document}
\maketitle

\begin{abstract}
General Hardy-Carleman type inequalities for Dirac operators are prov\-ed. New inequalities are derived involving particular traditionally used weight functions. In particular, a version of the Agmon inequality and Treve type inequalities are established. The case of a Dirac particle in a (potential) magnetic field is also considered. The methods used are direct and based on quadratic form techniques.
\end{abstract}


\section{Introduction}\label{sec:introduction}

In a recent work \cite{dolbeault-esteban-sere} (see also \cite{dolbeault-esteban-duoandikoetxea} and \cite{dolbeault-esteban-loss-vega}) of J. Dolbeault et al. it was proved a version of Hardy type inequality related to the Dirac operator describing the behaviour of a spin $1/2$ particle with non-zero  rest mass under the influence of an electrostatic potential $q$. Namely, for a given scalar potential $q$ satisfying 
$$q (x) \to 0, \quad | x | \to \infty,$$
$$- \frac{\nu}{| x |} - c_{1} \leq q (x) \leq   c_{2}  =  \sup_{x \in \reals^{3}}  q (x)$$
with (a parameter)  $\nu \in (0,1)$   and $c_{1}, c_{2} \in \reals,$ there exists a value $\mu$ in the interval $(- 1, 1)$  such  that the inequality 
\begin{equation}\label{eq:first formula}
\int_{\reals^{3}} q  \left| \varphi \right|^{2} \ud x \leq  \int_{\reals^{3}} \left( \frac{ | \sigma \cdot \nabla \varphi |^{2}}{1 + \mu + q}  + \left( 1 - \mu\right) | \varphi |^{2} \right) \ud x
\end{equation}
holds true for all functions $\varphi$ in the  Sobolev space  $W_{2}^{1} (\reals^{3}; \complex^{2}).$    In \eqref{eq:first formula} $\nabla \varphi$ denotes the distributional gradient of $\varphi,  \sigma = (\sigma_{1}, \sigma_{2}, \sigma_{3})$ being the triplet of Pauli matrices
$$\sigma_{1} = 
\left( \begin{array}{cc} 
 0 & 1 \\ 
1 & 0  \end{array} \right), 
\sigma_{2} = 
\left( \begin{array}{cc} 
 0 & - i \\ 
i & 0  \end{array} \right),
\sigma_{3} = 
\left( \begin{array}{cc} 
 1 & 0 \\ 
0 & - 1  \end{array} \right).$$

$\mu$ is in fact taken as the smallest eigenvalue of the Dirac operator
$$H = - i \alpha \cdot \nabla + \beta + q,$$
where $\alpha = (\alpha_{1}, \alpha_{2}, \alpha_{3})$   and the matrices  $\beta, \alpha_{j} \in M_{4 \times 4} (\complex),  j = 1, 2, 3,$   are defined as 
$$\sigma_{1} = 
\left( \begin{array}{cc} 
 0 & \sigma_{j} \\ 
\sigma_{j} & 0  \end{array} \right),
\beta =  
\left( \begin{array}{cc} 
 I d & 0 \\ 
0 & - I d  \end{array} \right).$$
($I d$ is the identity matrix in $\complex^{2}$). As a consequence, in the case of the Dirac-Coulomb Hamiltonian $H$, that is, when $q (x) \approx 1/| x |,$  a simple limiting argument yields to the following Hardy type inequality
\begin{equation}\label{eq:second formula}
\int_{\reals^{3}} \frac{| \varphi |^{2}}{| x |} \ud x  \leq  \int_{\reals^{3}}  \left(  \frac{| \sigma \cdot \nabla \varphi |^{2}}{ 1 + \frac{1}{| x |}}  + | \varphi |^{2} \right) \ud x, \quad  \varphi \in W_{2}^{1} (\reals^{3}; \complex^{2}),
\end{equation}
which can be interpreted also as  a {\em relativistic uncertainly principle} for $H$.

The mentioned inequalities \eqref{eq:first formula} and \eqref{eq:second formula} allow in a natural way to describe distinguished self-adjoint extensions of Dirac  operators with certain singularities of  the potentials \cite{esteban-loss}. It should be mentioned that they are also useful  in the study of spectral properties of Dirac operators  especially relevant to the problems of scattering theory as, in particular, to get information about the behaviour of the resolvent nearly to the continuous spectrum, in proving of the limiting absorption principle and others.

Our main purposes is to prove Hardy type inequalities for Dirac  operators in more general setting involving arbitrary weight functions. By a Dirac operator we mean a first order partial differential operator with constant coefficients of the form 
\begin{equation}\label{eq:dirac}
H = \sum_{j = 1}^{n} \alpha_{j} D_{j}  + \beta,
\end{equation}
where $D_{j}  = - i   \partial / \partial x_{j} \ \ (j = 1, ..., n), \ x =(x_{1}, ..., x_{n}) \in  \reals^{n}, $   $\alpha_{j}  \  (j = 1, ...,n)$  and $\beta$  are $m \times m$ Hermitian matrices which satisfy the Clifford's anticommutation relations
\begin{equation}\label{eq:clifford}
\alpha_{j}  \alpha_{k}  + \alpha_{k} \alpha_{j}   =  2 \delta_{j k} \ (j,k = 1, ..., n), \ \alpha_{j}  \beta   + \beta \alpha_{j}  = 0 \ \ (j = 1, ..., n), \ \beta^{2} = 1,
\end{equation}
$m = 2^{n/2}$  for $n$ even and $m = 2^{(n + 1)/2}$  for $n$ odd; $\delta_{jk}$ denotes the Kronecker symbol $(\delta_{j k} = 1 $\ \ if  \ \ $j = k$ and $\delta_{j k}  = 0 $  if $j \neq k$).  
The Dirac operator $H$ is usually considered acting in the space $L_{2}
(\reals^{n}; \complex^{m})$ defined on its maximal domain the Sobolev space
$W_{2}^{1}  (\reals^{n}; \complex^{m})$. However, to cover some
more general situations, also important  in applications or by themselves,
it will be also assumed that the operator $H$ is defined on an arbitrary
(open) domain $\Omega$ in $\reals^{n}$.   In those cases  as  a domain of
$H$ is certainly taken the Sobolev space of functions defined on $\Omega$.

Our aim is to describe conditions on the weight functions $a, b$ under which an  inequality  of the form 
\begin{equation}\label{eq:weight}
c   \| a u \| \leq  \| b H u \|, \quad u \in \mathcal{D},
\end{equation}
holds true for a suitable class of  functions $u \in \mathcal{D}, c$ being a positive  constant depending only on $a, b$ and, eventually, $\Omega$. We assume that  $a$ and $b$ are sufficiently smooth functions, it will be enough to consider $a, b$ to be of class $C^{2}$. In \eqref{eq:weight} $\| \cdot \|$ designates the  norm on $L_{2}
(\reals^{n}; \complex^{m})$. We call estimates like that in \eqref{eq:weight} as Hardy-Carleman type inequalities.  Inequalities of this kind related mostly to the Laplace operator are also named weighted  Hardy inequalities or weighted Poincar$\acute{e}$-Sobolev  inequalities or  weighted Friedrichs inequalities  as well. We emanate the classical Hardy  inequality for Dirichlet form (see, for instance, \cite{mazja} or \cite{davies99} and \cite{kufner-maligranda-persson} for a history on the subject and farther references),  and the remarkable inequality due to Carleman stated  in \cite{carleman} in connection with the unique continuation property for second order elliptic differential equations. Apart  from the already mentioned works (there is  a vast literature on the topic), we refer to \cite{amrein} (see also \cite{amrein2}), \cite{hormander1}, \cite{hormander2}, \cite{jerison}, \cite{jerison-kenig}, and the references quoted there. 

In spite of the fact that the Laplace and Dirac operators are closely connected with each other, however, they behave quite  differently, and the methods properly for the Laplace operator are no  longer applicable to  the Dirac operators case. Nevertheless, as is  shown in this paper, some of traditional  methods can be refined to be available  also for Dirac operators. Compared to the Laplace operator case, the subject concerning  Dirac operators, to the best of our knowledge,  has beed studied rather sparingly in the literature. We mention the work \cite{jerison} in which certain Carleman inequalities for the (massless) Dirac operator are established. 

The paper is organised as follows. In Section \ref{sec:general HC} we discuss general weight inequalities for the  Dirac operator defined by \eqref{eq:dirac} and \eqref{eq:clifford}. Conditions  on the weight functions $a, b$ are given in order that an inequality of the form \eqref{eq:weight} hold true. The proofs are based on quadratic form techniques. In Section \ref{sec:The case of radial weight functions} there is established a  Carleman inequality for the particular case, but useful in applications, of radial weight functions. Section \ref{sec:Example of Hardy-Carleman inequalities} contains concrete Carleman  type  inequalities  that are derived  from general results by handling special  traditionally used weight functions. In this  way a version of the Agmon inequality and Treve  type  inequalities are obtained as  particular cases of  the general inequalities. In Section \ref{sec:A Carleman type inequality (another approach)}  we prove an inequality with power like  weight functions for which the approach applied  previously is not available. The arguments in the proof of the corresponding Hardy-Carleman  inequality use eigenfunctions expansions by  involving spherical harmonic functions.  Finally,  in Section \ref{sec:Inequalities for the Dirac operator with a magnetic field} the results are extended  to the case of the Dirac operator describing  a relativistic particle  in a potential  magnetic field.

\section{General Hardy-Carleman inequalities}\label{sec:general HC}

In this section we discus general weight inequalities of the form \eqref{eq:weight}. In order to make use our method for the  proof it is always required that the  weight functions $a, b$ to be of class $C^{2}$. We describe  conditions under which \eqref{eq:weight} hold true for a suitable class of functions $u \in \mathcal{D}$.

It will be convenient to work in polar coordinates  $(r, \omega) \in (0, \infty) \times S^{n -1}$ :   $r  = | x |$,  $\omega  = x / | x | $ for $x \neq 0$.
Denoting $\omega_{j} = x_{j} / | x |$ $(j = 1, ..., n)$,  the coordinates
of $\omega$, we have
$$\frac{\partial}{\partial x_{j}}  =  \omega_{j}   \frac{\partial}{\partial r} + r^{- 1}   \Omega_{j} \quad (j = 1, ..., n),$$
where  $\Omega_{j}$   is a vector-field  on the unit sphere $S^{n - 1}$
satisfying
$$\sum_{j = 1}^{n} \omega_{j} \Omega_{j}  = 0, \quad \sum_{j = 1}^{n}  \Omega_{j} \omega_{j} = r \sum_{j = 1}^{n}
 \frac{\partial \omega_{j}}{\partial x_{j}}  = n - 1.$$
 Let
 $$\widehat{\alpha}  =  \sum_{j = 1}^{n} \alpha_{j} \omega_{j},$$
 then
$$H = \sum_{j = 1}^{n} \alpha_{j}  \left(- i  \left( \omega_{j} \frac{\partial}{\partial r} + r^{- 1} \Omega_{j} \right)\right) + \beta
=$$
$$=  - i  \sum_{j = 1}^{n} \alpha_{j} \omega_{j}  \frac{\partial}{\partial r} - i  r^{- 1} \sum_{j = 1}^{n}
\alpha_{j} \Omega_{j}   + \beta  = $$
$$=  - i  \widehat{\alpha} \frac{\partial}{\partial r} - i r^{- 1} \sum_{j = 1}^{n} \alpha_{j} \Omega_{j}  + \beta,$$ 
i.e.,
\begin{equation}\label{eq:H}
 H  =  - i \widehat{\alpha} \frac{\partial}{\partial r} - i r^{- 1} \sum_{j = 1}^{n} \alpha_{j} \Omega_{j}  + \beta.
\end{equation}
 It is  easy to see that
$$\widehat{\alpha}^{2} =  \left( \sum_{j = 1}^{n} \alpha_{j} \omega_{j}  \right)^{2} = \sum_{j = 1}^{n} \sum_{k = 1}^{n} \alpha_{j} \alpha_{k} \omega_{j} \omega_{k}  = 1,$$
 and
$$\widehat{\alpha} \sum^{n}_{j = 1} \alpha_{j} \Omega_{j} = \sum_{j = 1}^{n} \sum_{k = 1}^{n} \alpha_{j} \alpha_{k}   \omega_{j}  \Omega_{k}  = $$
$$= \sum_{j = 1}^{n}  \alpha_{j}^{2}  \omega_{j} \Omega_{j} + \sum_{j \neq k} \alpha_{j}  \alpha_{k} \omega_{j} \Omega_{k} = \sum_{j = 1}^{n} \omega_{j} \Omega_{j} + \sum_{j < k} \alpha_{j} \alpha_{k} \omega_{j} \Omega_{k} + \sum_{j > k} \alpha_{j} \alpha_{k} \omega_{j} \Omega_{k} =$$
$$= \sum_{j < k} \alpha_{j} \alpha_{k} \omega_{j} \Omega_{k} + \sum_{j < k} (- \alpha_{j} \alpha_{k} )
\omega_{k} \Omega_{j} =$$
$$=  \sum_{j < k} \alpha_{j}  \alpha_{k} (\omega_{j} \Omega_{k} - \omega_{k} \Omega_{j} ).$$

We let
$$L  =   \sum_{j < k} \alpha_{j} \alpha_{k} (\omega_{j}  \Omega_{k}  - \omega_{k}  \Omega_{j}),$$
 an operator acting only in the $\omega$ variables, then 
$$i  \widehat{\alpha}  H  =  i  \widehat{\alpha} \left( - i  \widehat{\alpha}  \frac{\partial}{\partial r} - i r^{- 1} \sum_{j = 1}^{n} \alpha_{j} \Omega_{j} + \beta\right) =$$
$$= \frac{\partial}{\partial r} + r^{- 1} L + i  \widehat{\alpha}  \beta,$$
i.e.,
\begin{equation}\label{eq:iH}
i  \widehat{\alpha} H = \frac{\partial}{\partial r} + r^{- 1} L + i \widehat{\alpha} \beta.
\end{equation}

Note that
$$L =  \sum_{j < k} \alpha_{j}  \alpha_{k} \left(x_{j} \frac{\partial}{\partial  x_{k}} - x_{k} \frac{\partial}{\partial  x_{j}}\right),$$
that follows immediately from the relations
$$ x_{j} \frac{\partial}{\partial x_{k}} - x_{k} \frac{\partial}{\partial x_{j}}  =  x_{j} \left(\omega_{k}   \frac{\partial}{\partial r} + r^{- 1} \Omega_{k}\right)
- x_{k} \left(\omega_{j}   \frac{\partial}{\partial r}   + r^{- 1} \Omega_{j}\right)  =$$
$$= r \omega_{j} \omega_{k} \frac{\partial}{\partial r} + \omega_{j} \Omega_{k} - r \omega_{k} \omega_{j}  \frac{\partial}{\partial r} - \omega_{k} \Omega_{j} = $$
$$= \omega_{j}  \Omega_{k}  -  \omega_{k}  \Omega_{j}.$$

Returning to the inequality \eqref{eq:weight} we put
$$v = b u$$
and observe that for any smooth function $\varphi$  one has
\begin{equation}\label{eq:commutator}
[ H,  \varphi ]  =   - i   \sum_{j = 1}^{n}  \alpha_{j}   \frac{\partial \varphi}{\partial x_{j}}, 
\end{equation}
where $[H, \varphi ]$ denotes the commutator of $H$ and $\varphi,$ viewing simultaneously
$\varphi$ as the multiplication operator by the function $\varphi.$

Using the obtained relation \eqref{eq:commutator} we can write
$$ b H = H b +  i \sum_{j = 1}^{n} \alpha_{j} \frac{\partial b}{\partial x_{j}}$$
or
$$b H b^{- 1}  = H + i B,$$
where
\begin{equation}\label{eq:B}
B = b^{-1} \sum_{j = 1}^{n}  \alpha_{j}   \frac{\partial b}{\partial x_{j}}.
\end{equation}

Thus, the inequality \eqref{eq:weight} reduces to the estimation  from the below of
the  quadratic form
$$h [v] =  \|  (H + i B) v \|^{2}$$
on suitable elements $v$.

Let us start with the case of the Dirac operator considered on $\reals^{n}$. Using the polar coordinates, by \eqref{eq:H} and \eqref{eq:iH}, $h[v]$ can be represented as follows
$$h [v] = \int_{0}^{\infty} \int_{S^{n - 1}} r^{n - 1} \left| \left(- i \widehat{\alpha}
  \frac{\partial}{\partial r} - i r^{- 1} \sum_{j = 1}^{n} \alpha_{j} \Omega_{j} + \beta + i B (r \omega) \right) v \left(r \omega\right) \right|^{2} \ud r \ud \omega =$$
$$=  \int_{0}^{\infty}  \int_{S^{n - 1}}   r^{n - 1}  \left| i \widehat{\alpha}  \left(- i \widehat{\alpha}
  \frac{\partial}{\partial r} - i r^{- 1} \sum_{j = 1}^{n} \alpha_{j} \Omega_{j} + \beta + i B \left(r \omega\right) \right) v \left(r \omega\right) \right|^{2} \ud r \ud \omega =$$
$$=  \int_{0}^{\infty} \int_{S^{n - 1}} r^{n - 1} \left| \left(\frac{\partial}{\partial r} + r^{- 1} L + i \widehat{\alpha}  \beta - \widehat{\alpha} B \left(r \omega\right) \right) v \left(r \omega\right) \right|^{2} \ud r \ud \omega.$$
Letting $r = e^{t}$  we have
$$\frac{\partial}{\partial r} = e^{- t}  \frac{\partial}{\partial t},$$
and then
$$h [v] =  \int_{- \infty}^{\infty} \int_{S^{n - 1}} e^{t (n - 2)} \left| \left(\frac{\partial}{\partial
t} + L + i \widehat{\alpha} \beta e^{t} - \widehat{\alpha} B \left(e^{t} \omega\right) e^{t} \right) v \left(e^{t} \omega\right) \right|^{2} \ud t \ud \omega.$$
In order to remove the exponent in the expression on right-hand side, we
let
$$\widetilde{v} (t, w) = e^{\nu t} v \left(e^{t} w\right) \quad with \quad  2 \nu = n - 2$$
and, taking into account that
$$\left(\frac{\partial}{\partial t}\right) e^{- \nu t} = - \nu e^{- \nu t} + e^{- \nu t} \frac{\partial}{\partial t} = e^{- \nu t} \left(\frac{\partial}{\partial t} - \nu\right),$$
we obtain
$$h [v] =  \int_{- \infty}^{\infty} \int_{S^{n - 1}} \left| \left(\frac{\partial}{\partial t} - \frac{n - 2}{2} + L + i \widehat{\alpha} \beta e^{t} - \widehat{\alpha} B \left(e^{t} \omega\right) e^{t} \right) \widetilde{v} \left(t, \omega\right) \right|^{2} \ud t
\ud \omega.$$

It follows
$$h [v] =  \int_{- \infty}^{\infty} \int_{S^{n - 1}} \left| \frac{\partial}{\partial t} \widetilde{v} \left(t, \omega\right) \right|^{2} \ud t \ud \omega   + $$
 $$ +    \int_{- \infty}^{\infty} \int_{S^{n - 1}} \left| \left( L  + i \widehat{\alpha}  \beta  e^{t} - \widehat{\alpha} B
\left(e^{t} \omega\right) e^{t}  - \frac{(n - 2)}{2}   \right) \widetilde{v} \left(t, \omega\right) \right|^{2} \ud t \ud \omega +$$
$$ +  \int_{- \infty}^{\infty} \int_{S^{n - 1}}  \left\langle \left( -  i \widehat{\alpha}  \beta  e^{t} + \widehat{\alpha}  \frac{\partial \left( B \left(e^{t} \omega\right) e^{t} \right)}{\partial t} \right) \widetilde{v}  \left(t, \omega\right),  \widetilde{v}  \left(t, \omega\right) \right\rangle \ud t \ud \omega$$
($\langle \cdot, \cdot \rangle$   denotes the scalar product in $\complex^{m}$). 
	
In  this way we obtain the following estimation
\begin{equation}\label{eq:scalar}
 h [v]  \geq  \int_{- \infty}^{\infty}  \int_{S^{n - 1}} \langle G(t, \omega) \widetilde{v} (t, \omega),  \widetilde{v} (t, \omega) \rangle \ud t \ud \omega,  
\end{equation}
where
$$ G (t, \omega) =  -i  \widehat{\alpha}  \beta  e^{t} + \widehat{\alpha} \frac{\partial \left(B\left( e^{t} \omega\right) e^{t} \right)}{\partial t}.$$ 

The form on the right-hand side of \eqref{eq:scalar} can be further transformed as
follows  (note that $\partial / \partial t  = r \partial / \partial r $)
$$\int_{- \infty}^{\infty}  \int_{S^{n - 1}} \langle G (t, \omega) \widetilde{v} (t, \omega), \widetilde{v} (t, \omega)
\rangle \ud t \ud \omega  = $$
$$=  \int_{- \infty}^{\infty}  \int_{S^{n - 1}} e^{2 \nu t} \langle G (t, \omega) v (e^{t}  \omega), v (e^{t}  \omega)
\rangle \ud t \ud \omega  =$$
$$=  \int_{0}^{\infty}  \int_{S^{n - 1}}   r^{n - 2}  \biggl\langle  \biggl(- i \widehat{\alpha} \beta r + \widehat{\alpha} r     \frac{\partial (B (r \omega) r )}{\partial  r} \biggl) v(r \omega), v (r \omega) \biggl\rangle r^{-1} \ud r \ud \omega  =$$
 $$=  \int_{0}^{\infty}  \int_{S^{n - 1}}   r^{n - 1}  \biggl\langle  r^{-1} \biggl(- i \widehat{\alpha}
 \beta  +  \widehat{\alpha} \frac{\partial (B(r \omega) r)}{\partial r} \biggl) b  (r
 \omega)^{2} u (r \omega), u (r \omega) \biggl\rangle \ud r \ud \omega.$$
   
 Now, wishing to involve the weight function $a$ we  require that the matrix-valued function
\begin{equation}\label{eq:M}
M (r, \omega) =  r^{-1}  \biggl(- i \widehat{\alpha} \beta + \widehat{\alpha} \frac{\partial (B(r \omega) r)}{\partial  r}\biggl) b (r
 \omega)^{2} a (r \omega)^{-2} 
\end{equation}
is positive definite (with respect to  the quadratic forms) uniformly with respect to $r$ and $\omega$, that in fact means
\begin{equation}\label{eq:M1}
 \langle  M (r, \omega)  u (r \omega),  u (r \omega)  \rangle \geq c | u (r \omega) |^{2} 
 \end{equation}
for a positive constant $c$ independent of $u, r$  and $\omega$.
Under this condition  there holds
$$h [ v ] \geq c \int_{0}^{\infty}  \int_{S^{n - 1}} r^{n - 1}  | a (r \omega) u (r \omega) |^{2} \ud r \ud \omega =
c \int_{\reals^{n}} | a(x) u(x) |^{2}  \ud x,$$
 that  leads  to the desired inequality \eqref{eq:weight}.    
For, as is seen, $\mathcal{D}$ can be taken the set of all  functions $u$ belonging to the Sobolev space $W_{2}^{1} (\reals^{n}; \complex^{m})$ having compact support in  $\reals^{n} \backslash  \{ 0 \}$.  

We summarize the above discussion in the following theorem.

\begin{thm}\label{thm:result1}
Let $a, b$ be positive functions of the  class $C^{2}$  for which  the condition \eqref{eq:M1}
is fulfilled. Then for any function $u$ belonging to the Sobolev space $W_{2}^{1} (\reals^{n}; \complex^{m})$
and having  a compact support in the set $\reals^{n} \backslash  \{ 0 \}$  the inequality \eqref{eq:weight} holds  true with a positive constant $c$ depending only on $a$ and $b$. 
\end{thm}
\begin{rem}\label{rem:rem1}   
The above arguments remain valid for the Dirac operator $H$ considered on an arbitrary open domain $\Omega$. The functions $u$ in \eqref{eq:weight} must be then taken in the Sobolev space $W_{2}^{1} (\Omega; \complex^{m})$ having compact supports in the set $\Omega$, respectively, in $\Omega \backslash \{ 0 \}$ if $0 \in \Omega.$
Incidentally, the condition \eqref{eq:M1} should be fulfilled on $\Omega$. Note that the condition \eqref{eq:M1} can be interpreted as that determining $\Omega$ on which the inequality \eqref{eq:weight} holds true.
\end{rem}
The inequalities like one as in \eqref{eq:weight} can be called {\em general \ Hardy-Carleman  inequalities} or, simply, Carleman type Inequalities.
Various  Carleman type  inequalities can be derived by choosing suitable weight functions $a, b$ defined on a domain $\Omega$ in $\reals^{n}$. In the remainder of this section we confine ourselves  to make  some remarks  still concerning on general  situations of Carleman inequalities. Concrete Carleman type inequalities and further
remarks will be given in the next sections.

Let the weight  functions $a, b$ be chosen  satisfying
\begin{equation}\label{eq:weight b}
 b (x) =  | x |^{1/2}  a (x), \quad \quad  x \in \Omega
\end{equation}
for a given (open) domain $\Omega$ in $\reals^{n}$. Then
\begin{equation}\label{eq:M2}
M (r, \omega) = - i   \widehat{\alpha}  \beta +  \widehat{\alpha}  \frac{\partial (B ( r \omega) r)}{\partial r}
\end{equation}
 and,  in order to establish a Carleman inequality for  this  case, we have to look that the  matrix-valued function  $M(r, \omega)$ given by \eqref{eq:M2} to be positive definite uniformly on $\Omega$. To this end it should be noted that the matrix $- i \widehat{\alpha} \beta$ is symmetric and has only two eigenvalues $\pm 1$ (the point is that, additionally, $(- i \widehat{\alpha} \beta)^{2} = 1$  and $- i   \widehat{\alpha}  \beta  \neq  \pm  1$).
  It is clear that the condition \eqref{eq:M1} will be fulfilled if the eigenvalues of the matrix
$$M_0 (r, \omega) = \widehat{\alpha} \frac{\partial (B (r \omega ) r)}{\partial r}$$
are situated on the right-side of 1, and if
\begin{equation}\label{eq:d}
  d = \inf_{r, \omega} (\lambda_{min} (r,\omega) -1) >0,
\end{equation}
where $\lambda_{min}  (r, \omega ) $   is the least  of  eigenvalues of $M_{0}
(r, \omega).$  Obviously,  $\lambda_{min}  (r, \omega )$ can be chosen
depending  continuously on $r$ and $\omega$, provided that the function
$b$  is of the class $C^{2}.$

Thus, we can formulate the following.

\begin{cor}\label{cor:cor1}    
Let $a, b$ be functions of the class $C^{2}$ satisfying \eqref{eq:weight b} on a given open domain $\Omega$  in $\reals^{n},$ and suppose that the condition \eqref{eq:d} is fulfilled. Then, for any function $u$ in the Sobolev  space  $W_{2}^1 (\Omega; \complex^{m})$ having its compact support in the set $\Omega \backslash \{ 0 \}$,
there  holds the following inequality
\begin{equation}\label{eq:cor1}
c \int_{\Omega}  | a (x) u (x)  |^{2}  \ud x   \leq  \int_{\Omega} |x|   | a(x)  H u(x)  |^{2}  \ud x  
\end{equation}
with a positive constant $c$ which can be taken equal to $d$ ($d$ being
defined by \eqref{eq:d}).
\end{cor}
\begin{rem}\label{rem:rem2}   
In case the domain $\Omega$ is bounded the factor $|x|$ in \eqref{eq:cor1}
can be omitted by changing suitably the constant $c$. 
\end{rem}

\section{The case of radial weight functions}\label{sec:The case of radial weight functions}

Throughout this section we suppose that the weight functions $a$,
$b$ depend only on the radial coordinate $r$, $r = |x|$. For this case the
conditions like \eqref{eq:M1}, being crucial for the fulfilment of a desired Hardy-Carleman
type inequality, became considerably simpler. So, if $b=b(r)$ depends only
on the radial coordinate $r$, then
$$\frac{\partial b}{\partial x_j} = b' (r) \frac{\partial r}{\partial x_j} = b' (r) \omega_j,$$
and, by \eqref{eq:B}, one has
$$B(r) = b(r)^{-1} \sum_{j=1}^n \alpha_j b'(r) \omega_j =b (r)^{-1} b'(r)    \widehat{\alpha},$$
i.e.,
$$B(r) =  b (r)^{-1} b'(r) \widehat{\alpha}.$$
Hence,
$$M_{0} (r, \omega)  =  \widehat{\alpha}  \frac{\partial (B (r, \omega) r)}{\partial r} = (b (r)^{- 1}  b^{'} (r) r)^{'} , $$
i.e.
 $$M_{0} (r, \omega)   = (b (r)^{- 1}  b^{'} (r) r)^{'}, $$
 Thus  the matrix-valued function
 $M_{0} (r, \omega) $  reduces in fact to a scalar function depending only on
 $r$.  If,   in addition,  the weight functions  are connected   between
 themselves  by the  relation  \eqref{eq:weight b}, then the condition \eqref{eq:d} becomes as follows
\begin{equation}\label{eq:c}
c : = \inf_{r} ((b (r)^{- 1} b^{'} (r)  r )^{'}  - 1 )  >  0,
\end{equation}
which, certainly, ensures the fulfilment of the inequality \eqref{eq:weight} with the constant $c$ determined as above.

A  particular case  of  just  mentioned  inequality can be  obtained
by taking
$$a (x)  =  | x |^{-1/2}  e^{\tau \varphi (x)} \quad and \quad b(x) =   e^{\tau \varphi (x)}, \quad x \neq 0$$
with $\tau > 0$ as a parameter, and $\varphi$ being a function of the class $C^{2}$ depending only on the radial coordinate $r$. For this case we have
$$(b (r)^{- 1}  b^{'} (r) r)^{'} = (e^{- \tau \varphi (r) } \tau \varphi^{'} (r) e^{\tau \varphi (r)} r )^{'}  =$$
$$=  \tau  (\varphi^{'} (r)  r)^{'}  =   \tau  ( \varphi^{''} (r) r  +  \varphi^{'} (r)),$$
i.e.,
$$(b (r)^{- 1}  b^{'} (r) r)^{'}  =  \tau  ( \varphi^{''} (r) r  +  \varphi^{'} (r)). $$ 

Now, it is clear that the condition \eqref{eq:c} is verified by assuming $\tau \gamma -1 > 0,$  where
\begin{equation}\label{eq:gamma}
\gamma = \inf ( \varphi '' (r) r +\varphi ' (r)) > 0 \quad ({on} \  \Omega)
\end{equation}

Thus, the following assertion can be made.

\begin{thm}\label{thm:result2}   
Let $\varphi$ be a function of class $C^{2}$ depending only on the
radial coordinate $r$ and satisfying \eqref{eq:gamma} on a given open domain $\Omega$ in
$\reals^n$. Then, for any function $u$ in the Sobolev space $W_2^1
(\Omega; \complex^m)$ having its compact support in $\Omega \setminus  \{ 0 \}$, the following Carleman type inequality
$$c \int_{\Omega} |x|^{-1} e^{2 \tau \varphi (x)} |u(x)|^2 \ud x \leq \int_{\Omega}  e^{2 \tau \varphi (x)} | H u(x)|^2 \ud x$$
holds true for $\tau > \gamma^{-1}$  (in particular,  for  sufficiently large $\tau$)  and a  positive constant $c$  depending only on $\tau$ and $\varphi$.
\end{thm}

\section{Example of Hardy-Carleman inequalities}\label{sec:Example of Hardy-Carleman inequalities}

In this section we derive concrete Hardy-Carleman inequalities by handling special frequently encountered weight functions. We restrict ourselves to consider the Dirac operator describing a relativistic particle with negligible mass. In this case the term $H$ containing $\beta$ is absent. In order to distinguish this special case the Dirac operator will be denoted by $H_{0}$, so
 $$H_{0}  =   \sum_{j = 1}^{n}  \alpha_{j}  D_{j}$$
 with all attributed conditions as in general case.
   For  the sake  of simplicity,  in  what follows,  we will always
 consider the operator $H_{0}$   on the  whole space $\reals^{n}$,   that
  is  acting   in the space  $L_{2} (\reals^{n}; \complex^{m})$  on its
  domain  the Sobolev space  $W_{2}^{1} (\reals^{n}; \complex^{m}).$
 
 Hardy-Carleman type  inequalities   will  be established  for  the  operator
  $H_{0}$ by  choosing suitable  weight  functions   depending    only on
  radial   coordinate $r$.   Under  the hypotheses  made   above  the
  matrix-valued function  $M(r, \omega)$,  as it was  already  mentioned   before,
  reduces to a scalar function  depending only on
  $r$, we denote it by $M(r)$. Namely (cf. Section \ref{sec:The case of radial weight functions}),
  $$ M(r)= r^{-1} (b(r)^{-1} b' (r) r)' b(r)^2 a(r)^{-2}.  $$

\textbf{Example 4.1.}  
Now, letting
 $$ b(x)= (1+ |x|^{2})^{\tau/2}, \qquad \tau >0, $$
 we have
   $$ r^{-1} (b(r)^{-1} b' (r) r)' b(r)^2= 2 \tau ( 1+r^2)^{\tau -2}, $$
 If we take
 $$ a(r) = ( 1+r^2)^{(\tau -2)/2}, \qquad r > 0, $$
 we obtain that
 $$ M(r)= 2 \tau >0. $$

 Thus we have proved the following inequality
 
 \begin{equation}\label{eq:example1}
   2 \tau \int_{\reals^n} (1+ |x|^2)^{\tau - 2} |u (x)|^2 \ud x \leq \int_{\reals^n} (1+ |x|^2)^{\tau} | H_0 u(x) |^2 \ud x,
 \end{equation}
 
 for all $\tau >0$.
 
 In \eqref{eq:example1} and in all considered  further inequalities as well, it is
assumed  that the function $u$ belongs to the Sobolev  space  $W_2^{1}  (
\reals^{n};   \complex^{m} )$  having  compact  support  in  the set
$\reals^{n}  \backslash \{ 0 \}.$

 The following particular cases, namely
 \begin{equation}\label{eq:ex2}
   2 \int_{\reals^n} (1+ |x|^2)^{ - 1} |u (x)|^2 \ud x \leq \int_{\reals^n} (1+ |x|^2) | H_0 u(x) |^2 \ud x,
 \end{equation}
 for $ \tau = 1$, and
\begin{equation}\label{eq:ex3}
      4 \int_{\reals^n} |u (x)|^2 \ud x \leq \int_{\reals^n} (1+ |x|^2) | H_0 u(x) |^2 \ud x, 
\end{equation}
 for $\tau = 2$, are important in applications and by themselves.
 The inequality \eqref{eq:ex2} can be named as an {\em Agmon type inequality} (cf. \cite{agmon})
 whereas \eqref{eq:ex3} as a {\em Hardy type inequality}  for the Dirac operator $H_{0}$.

\textbf{Example 4.2.} 
Next, we consider
 $$ b(x)= e^{\tau | x |^{\alpha}/2}, \quad \tau > 0, \ \alpha \in \reals, \  \alpha \neq 0.$$

 In this case
 $$ b(r)= e^{\tau r^{\alpha} /2} , \quad b'(r)= ( \tau /2) \alpha r^{\alpha - 1} e^{\tau r^{\alpha}/2}, $$
 and
  $$ r^{-1} (b(r)^{-1} b' (r) r)' b(r)^2= (\tau /2) \alpha^2 r^{\alpha-2} e^{\tau r^{\alpha}},$$
from which it is seen that it can be taken
$$ a(r) = r^{(\alpha -2)/2} e^{\tau r^{\alpha}/2}, \quad r>0.$$
Then
$$ M(r) =  \tau \alpha^2/ 2 >0, $$
and, thus, we obtain  the following inequality
 \begin{equation}\label{eq:example2} 
 (\alpha^2 \tau /2) \int_{\reals^n} |x|^{\alpha - 2} e^{\tau |x|^{\alpha}}
 |u (x)|^2 \ud x  \leq \int_{\reals^n} e^{\tau |x|^{\alpha}}  | H_0 u(x) |^2 \ud x, 
\end{equation}
 for $\tau >0$  and   $\alpha \in \reals\setminus \{0 \}.$

The following useful inequality
\begin{equation}\label{eq:ex4}
   (\tau /2) \int_{\reals^n} |x|^{ - 1} e^{\tau |x|} \ |u (x)|^2 \ud x  \leq \int_{\reals^n} e^{\tau |x|} \  | H_0 u(x) |^2 \ud x, 
\end{equation}
 is a particular case of \eqref{eq:example2} for $ \alpha = 1$.

The inequality \eqref{eq:example2}  for $\alpha = 2$ corresponds to the following one
\begin{equation}\label{eq:ex5}
  2 \tau \int_{\reals^n}  e^{\tau |x|^{2}} \
 |u (x)|^2 \ud x  \leq \int_{\reals^n} e^{\tau |x|^{2}} \  | H_0 u(x) |^2 \ud x, 
\end{equation}
 which can be called as a  {\em Treve type inequality } for the Dirac operator $H_0$. We cite \cite{treves} for related inequalities involving differential operators.
 
\textbf{Example 4.3.} 
Finally, let us  consider the weight function
 $$b (x)   =   e^{\tau (\log | x |)^{2} /2 },  \quad  \tau > 0.$$
  
We have
$$ b (r)   =   e^{\tau (\log r)^{2} /2 },  \quad  b^{'} (r) =  \tau  r^{- 1}  (\log r) e^{\tau (\log r)^{2} /2 },$$
hence
$$r^{- 1} (b (r)^{- 1} b^{'} (r)  r )^{'} b (r)^{2}  =  \tau  r^{- 2}   e^{ \tau (\log r)^{2}}.$$

If it is taken
$$a (r)  =  r^{- 1}   e^{ \tau (\log r)^{2} /2}, $$
then
$$M (r)  = \tau,$$
and,  thus,  we have  proved the following inequality
\begin{equation}\label{eq:example3}
 \tau   \int_{\reals^{n}}  | x |^{- 2}  e^{ \tau (\log | x |)^{2}} | u (x) |^{2}  \ud x
 \leq   \int_{\reals^{n}}    e^{ \tau (\log | x |)^{2}} |  H_{0} u (x) |^{2}  \ud x 
\end{equation}
 for $\tau > 0.$

\begin{rem}\label{rem:rem3}
A related inequality to \eqref{eq:example3} was proved in \cite{jerison}  (cf. \cite{jerison}, Theorem 2). However, in \cite{jerison} instead of the  whole space $\reals^{n}$ is taken  a domain $\Omega = \{  x \in \reals^{n} : a < | x | < b \}$ assuming $0 < a < b < 1$. In \cite{jerison} it is in fact proved the following one
\begin{equation}\label{eq:rem3}
\| e^{\tau \varphi}  u \|_{L^{q} (\Omega; \complex^{m})}   \leq C  \| e^{\tau \varphi} H_{0}  u \|_{L^{2} (\Omega; \complex^{m})} \quad for \quad all \quad  u \in C^{\infty}_{0}  (\Omega; \complex^{m}),
\end{equation}
where $\varphi (x) =  (\log | x |)^{2}/2,$  $\tau > 0, q =  (6 n - 4)/(3 n - 6),$ and $C$ depending only on $a, b$, and $n$.   By applying  our arguments an inequality like \eqref{eq:rem3} follows  for $q = 2$ as well, but with a constant $C$  depending on $a, b, n$  and also $\tau.$
\end{rem}
 
\section{A Carleman type inequality (another approach)}\label{sec:A Carleman type inequality (another approach)}

In this section we study the following Carleman type inequality
\begin{equation}\label{eq:carleman}
c   \int_{\reals^{n}}  | x |^{\tau}   | u (x) |^{2} \ud x  \leq \int_{\reals^{n}}  | x |^{\tau + 2}   | H_{0} u (x) |^{2} \ud x, \quad \tau \in  \reals,
\end{equation}
for the Dirac operator $H_{0}$. Recall that $H_{0}$ denotes the Dirac
operator  for  the case  of a particle with  negligible mass.
 
It is easily seen  that for the  weight functions
$$a (x)  =  | x |^{\tau/2}, \quad b (x) =  | x |^{(\tau + 2)/2},$$
as  in \eqref{eq:carleman}, the function $M (r)$, defined as in the previous section, is identically null.

In this connection the results discussed  in previous sections cannot be
applied to obtain an inequality with such weight functions.  On the other
hand, it seems that  the inequality \eqref{eq:carleman}  in general  fails. The following
is true however.
\begin{thm}\label{thm:result3}    
Let $n > 1$   and  let $\tau$  be a real number such that  $\tau \neq 2k -
n$ for integers $k \in \mathbf{Z}.$ Then, the inequality \eqref{eq:carleman} holds for
any  function  $u$ in the  Sobolev space  $W_{2}^{1} (\reals^{n};
\complex^{m})$  having  compact   support  in the  set $\reals^{n}
\backslash \{ 0 \}$  with a positive constant $c$ depending only   on  $d : =
\min_{k \in \mathbf{Z}} | \tau  +  n - 2 k |.$
\end{thm}
\begin{proof}     
It will be convenient to pass in \eqref{eq:carleman} to polar coordinates. We have
$$c  \int_{0}^{\infty}  \int_{S^{n - 1}}   r^{n - 1} r^{\tau}  | u (r \omega) |^{2}  \ud r \ud \omega \leq$$
$$\leq\int_{0}^{\infty}  \int_{S^{n - 1}}   r^{n - 1} r^{\tau + 2} \left| \left(   -i \widehat{\alpha}  \frac{\partial}{\partial r}  - i r^{- 1} \sum_{j = 1}^{n}  \alpha_{j} \Omega_{j} \right)  u \left(r \omega\right) \right|^{2} \ud r \ud \omega,$$ 
or equivalently,
$$c  \int_{0}^{\infty}  \int_{S^{n - 1}}   r^{n - 1} r^{\tau}  | u (r \omega) |^{2}  \ud r \ud \omega    \leq
\int_{0}^{\infty}  \int_{S^{n - 1}}   r^{n - 1} r^{\tau + 2}   \left|
\left(\frac{\partial}{\partial r}  +  r^{- 1}   L \right)  u \left(r \omega\right) \right|^{2} \ud r \ud \omega.$$

Further we let  $r = e^{t}.$ Then
$$\frac{\partial}{\partial r}  =  e^{- t}  \frac{\partial}{\partial t}, \quad \frac{\partial}{\partial r}  +   r^{- 1} L =  e^{- t}  \left(\frac{\partial}{\partial t} + L\right), $$
and we have
$$c  \int_{-\infty}^{\infty}  \int_{S^{n - 1}}   e^{(n - 1) t} e^{\tau t}  | u (e^{t} \omega) |^{2} e^{t} \ud t \ud \omega
\leq$$
$$\leq  \int_{- \infty}^{\infty}  \int_{S^{n - 1}} e^{(n - 1)t} e^{(\tau + 2) t} e^{- 2 t} \left| \left( \frac{\partial}{\partial t}  + L \right)  u (e^{t} \omega) \right|^{2} e^{t}  \ud t \ud \omega,$$    
 i.e.,
$$c  \int_{-\infty}^{\infty}  \int_{S^{n - 1}}   e^{(\tau + n) t}   | u (e^{t} \omega) |^{2}  \ud t \ud \omega    \leq
\int_{- \infty}^{\infty}  \int_{S^{n - 1}}   e^{(\tau + n)t}   \left|
\left(\frac{\partial}{\partial t}  +   L \right)  u (e^{t} \omega) \right|^{2} \ud t \ud \omega.$$ 
To remove the exponents denote
$$v (t, \omega)  =  e^{\nu t}   u (e^{t}  \omega) \quad with \quad 2 \nu  =  \tau + n.$$

We have
$$ \left(\frac{\partial}{\partial t}\right)  e^{- \nu t}   =  - \nu e^{- \nu t} + e^{- \nu t} \frac{\partial}{\partial t}  =  e^{- \nu t} \left(\frac{\partial}{\partial t} - \nu\right),$$
    i.e.,
$$\left(\frac{\partial}{\partial t}\right)  e^{- \nu t}   =   e^{- \nu t} \left(\frac{\partial}{\partial t} - \nu\right), $$
 and the inequality becomes
\begin{equation}\label{eq:inequality}
 c  \int_{-\infty}^{\infty}  \int_{S^{n - 1}}  | v (t, \omega) |^{2}  \ud t \ud \omega    \leq
\int_{- \infty}^{\infty}  \int_{S^{n - 1}}  \left| \left(\frac{\partial}{\partial t} - \nu + L \right) v  (t, \omega) \right|^{2} \ud t \ud \omega.
\end{equation}

As is easily seen, it is sufficient to  check  the  obtained  inequality for  functions of the form
$$v (t, \omega) =  f (t)  v_{k}  (\omega),$$
where  $v_{k}$  are  eigenfunctions  (spherical functions)  corresponding to
the  eigenvalues of the operator $L$, i.e.,
$$L v_{k}   = k  v_{k}.$$
Recall  that
$$\sigma (L)   \subset   \mathbf{Z},$$
and that
$$L (L + n - 2) =  - \Delta_{\omega},$$
where $ - \Delta_{\omega}$ denotes the Laplace-Beltrami  operator of the sphere $S^{n - 1}.$  We cite \cite{stein} for the details  concerning spectral  properties of  the operator $\Delta_{\omega}$.

It can be supposed that
$$ \int_{S^{n - 1}}  | v_{k} ( \omega) |^{2}  =  1.$$
Then \eqref{eq:inequality} becomes
\begin{equation}\label{eq:ineq}
c  \int_{-\infty}^{\infty} | f (t) |^{2}  \ud t \leq  \int_{-\infty}^{\infty}
\left| \left( \frac{\partial}{\partial t} - \nu + k \right)  f (t) \right|^{2} \ud t.
\end{equation}
In terms of Fourier transform the inequality \eqref{eq:ex3} is written as follows
$$c  \int_{-\infty}^{\infty} | \widehat{f} (\xi ) |^{2}  \ud \xi \leq
\int_{-\infty}^{\infty} |(i \xi  - \nu + k) \widehat{f} (\xi ) |^{2} \ud \xi,$$
where
$$\widehat{f} (\xi )   =  \frac{1}{\sqrt{2 \pi}}   \int_{-\infty}^{\infty}  f (t)  e^{-i t \xi}  \ud t.$$
The last  inequality reduces to the following estimate
$$c \leq | \  i \xi - \nu + k \ |^{2}.$$
It can  be taken
$$c \leq  (\tau  + n  - 2 k)^{2} / 4,$$
provided that
$$| i \xi - \nu + k |^{2}  = \xi^{2}  + (\nu  - k)^{2}  \geq  (\nu  - k)^{2}   =  (\tau  + n  - 2 k)^{2} / 4.$$
 This completes the proof.  
\end{proof}

\section{Inequalities for the Dirac operator with a magnetic field}\label{sec:Inequalities for the Dirac operator with a magnetic field}
 
 Let $H_{A}$  denote the Dirac operator with a magnetic field
 $$H_A = \sum_{j=1}^n \alpha_j ( D_j - A_j (x) ) + \beta,$$
where  $A(x)=(A_1(x), \ldots, A_n(x))$ is a vector potential describing the
  magnetic field.
 Assume that $A$ is a smooth vector-valued function with its components
  $A_j$ sufficiently rapidly decreasing (at infinity) functions in order to
  preserve the same domain the Sobolev space
  $W_2^1 ( \reals^n; \complex^m )$
  (or, respectively $W_2^1 ( \Omega; \complex^m )$
  if it is confined on a domain $\Omega$ in $ \reals^n$)
  as for the corresponding free Dirac operator.

For special classes of magnetic fields, but sufficiently large and
important for applications, weighted estimates like those discussed in  the
previous sections, can be reduced to the usual case of the free Dirac
operator.     So, let the magnetic potential $A$ be of the form
$$ A = \nabla \varphi $$
 ($\nabla = ( \partial / \partial x_1, \ldots, \partial / \partial x_n )$
 denotes the gradient operator), where $\varphi$ is a real-valued function
 possessing required properties in accordance with those of the magnetic
 field $A$.

Now, let an inequality of the form \eqref{eq:weight}, i.e.,
$$ c \  \| a u \| \leq \| b H u \|, \quad u \in D, $$
holds true. Then there holds the following one
$$c \  \| e^{i \varphi} a u \| \leq \| e^{i \varphi} b H u \|, \quad u \in D, $$
provided that $ | e^{i \varphi} | =1$ (it was supposed that $\varphi$ is a
real-valued function).

Denoting $$v = e^{i \varphi} u ,$$
the last inequality becomes
$$  C \  \| a   v \| \leq \| b \  e^{i \varphi} H e^{-i \varphi} v \|$$
for $v= e^{i \varphi} u $ with $u \in D$.

According to the relation \eqref{eq:commutator} we can write
$$ [H, e^{i \varphi}] = - i \sum_{j=1}^n \left(i e^{i \varphi} \frac{\partial \varphi}{\partial x_j} \right) \alpha_j =
e^{i \varphi} \sum_{j=1}^n \left( \frac{\partial \varphi}{\partial x_j} \right) \alpha_j , $$
i.e.,
$$ H  e^{i \varphi} = e^{i \varphi} H + e^{i \varphi} \sum_{j=1}^n \left( \frac{\partial \varphi}{\partial x_j} \right) \alpha_j ,$$
or, what is the same,
$$ e^{i \varphi} H  e^{-i \varphi} = H - \sum_{j=1}^n \left( \frac{\partial \varphi}{\partial x_j} \right) \alpha_j .$$
But
$$ A = \nabla \varphi $$
is equivalent to
$$ \sum_{j=1}^n  A_j (x) \alpha_j = \sum_{j=1}^n \left( \frac{\partial \varphi}{\partial x_j} \right) \alpha_j.$$
This last fact can be easily explained by using the anticommutation properties of
the matrices $\alpha_j$ $(j=1, \ldots, n)$.

So,
$$ e^{i \varphi} H  e^{-i \varphi} = H - \sum_{j=1}^n A_j (x) \alpha_j = H_A,
$$
i.e.
$$ e^{i \varphi} H  e^{-i \varphi} = H_A,  $$
and, thus, we obtain an inequality
\begin{equation}\label{eq:ab}
c \ \| a  v \| \leq \| b H_A v \|
\end{equation}
for the Dirac operator $H_A$ with the same weight functions $a$, $b$ and the
constant $c$ as in \eqref{eq:weight}.
\begin{thm}\label{thm:result4}
Under the above hypotheses  suppose that the weight functions $a, b$  are of class $C^{2}$ satisfying the condition \eqref{eq:M1}. Then an inequality \eqref{eq:ab} 
 for the Dirac operator $H_{A}$  holds true for all  functions $v$ belonging to the  Sobolev space $W_{2}^{1} (\reals^{n}; \complex^{m})$  and  having   compact supports in the set $\reals^{n}\setminus \{ 0 \}$ with a positive constant $c$ depending  only on $a$ and $b$.
\end{thm}

In view of the discussion undertaken  above the other results mentioned previously can be extended in  obvious fashion to the case  of the Dirac operator $H_{A}$.


\section*{Acknowledgments}

The author wishes to express her gratitudes to Professor Ari Laptev for fruitful discussions on the topic.

\bibliography{Hardy-Carleman}
\bibliographystyle{alpha}

\end{document}